\documentclass[12pt,a4paper,reqno]{amsart}
\usepackage{amsthm,amsfonts,amsmath,amssymb}
\usepackage{mathrsfs}
\usepackage[top=1.6in, bottom=1.3in, left=1.3in, right=1.3in]{geometry}
\usepackage{hyperref}
\usepackage{enumerate}

\author{Abhishek Jha}
\address{Indraprastha Institute of Information Technology, New Delhi, India}
\email{abhishek20553@iiitd.ac.in}



\title[On the G.C.D. of polynomial maps and their indices]
{On terms in a dynamical divisibility sequence having a fixed G.C.D with their indices}

\theoremstyle{plain}
\newtheorem{theorem}{Theorem}[section]
\newtheorem{lemma}[theorem]{Lemma}

 \newtheorem{proposition}[theorem]{Proposition}
\theoremstyle{definition}
\newtheorem{definition}[theorem]{Definition}
\newtheorem{remark}[theorem]{Remark}
\numberwithin{equation}{section}

\newcommand{\ord}{\mathfrak{o}_{F}} 
\renewcommand{\l}{\ell_{F}}
\renewcommand{\r}{{I}_{F}}

\newcommand{\OO}{\mathcal O}
\newcommand{\e}{\varepsilon}
\newcommand{\lcm}{\mathrm{lcm}}
\newcommand{\Mod}[1]{\ (\mathrm{mod}\ #1)}

\renewcommand{\t}{\text{dynamically deficient}}

\begin{document}

\maketitle

\begin{abstract}
Let $F$ and $G$ be integer polynomials where $F$ has degree at least $2$. Define the sequence $(a_n)$ by $a_n=F(a_{n-1})$ for all $n\ge 1$ and $a_0=0.$ Let $\mathscr{B}_{F,\,G,\,k}$ be the set of all positive integers $n$ such that $k\mid \gcd(G(n),a_n)$ and if $p\mid \gcd(G(n),a_n)$ for some $p$, then $p\mid k.$
Let $\mathscr{A}_{F,\,G,\,k}$ be the subset of $\mathscr{B}_{F,\,G,\,k}$ such that $\mathscr{A}_{F,\,G,\,k}=\{n\ge 1 : \gcd(G(n),a_n)=k\}$. In this article, we prove that the asymptotic density of $\mathscr{A}_{F,\,G,\,k}$ and $\mathscr{B}_{F,\,G,\,k}$ exists for a class of $(F,G)$ and also compute the explicit density of $\mathscr{A}_{F,\,G,\,k}$ and $\mathscr{B}_{F,\,G,\,k}$ for $G(x)=x.$
\end{abstract}

\section{Introduction}

Let $F(x)\in \mathbb{Z}[x]$ be a polynomial with positive leading coefficient and degree at least $2,$ then we define the sequence $(a_n)$ recursively as $a_n=F(a_{n-1})$ and $a_0=0.$ That is, $a_n$ denotes the 
$n$-th composition of $F$ with itself evaluated at 0. If $(a_n)$ contains finitely many distinct terms,  then $0$ is said to be a $\textsl{preperiodic point}$, which means that there exist distinct integers $m$ and $n$ such that $a_m=a_n.$ If $(a_n)$ contains infinitely many distinct terms, then $0$ is a $\textsl{wandering point}.$ For the rest of the article we assume that $0$ is a $\textsl{wandering point}$ or that the sequence is unbounded.

\noindent One can see that the sequence $(a_n)$ is a divisibility sequence, i.e., $a_n\mid a_m\text{ for }n\mid m.$ Such dynamical sequences share many characteristics with their well-known cousins: sequences arising from algebraic groups such as Lucas sequences and elliptic divisibility sequences. Divisibility sequences such as Lucas sequences or elliptic divisibility sequences have been recently studied, and relations between these sequences and their indices are of particular interest. For example, index divisibility sets have been investigated in the case of linear recurrences \cite{aa} or elliptic divisibility sequences \cite{jh}, and particular characterizations for such sets have been discovered. The authors in \cite{sa} have considered the case where the $n$-th term in a Lucas sequence shares a fixed gcd with $n$. Analogous results for elliptic divisibility sequences have been investigated in \cite{sk}. In the dynamical setting, the index divisibility set for the polynomial $x^{d}+c\in \mathbb{Z}[x]$ was studied in \cite{cyclic1}. In this survey \cite{et}, a complete account of these results is mentioned. 

Authors in \cite{ug} ask if the methods employed in \cite{sk} or \cite{sa} based on the rank of apparition function for such sequences could be translated to the dynamical setting to get more concrete results. Building on that, here in this article we study the $\gcd$ relations of these dynamical divisibility sequences with their indices based on the crucial rank of apparition function, which we define in the following.
\begin{definition}\label{df9}
We define $\ord(n),$ the rank of apparition of $n$ in $(a_{r}),$ to be the period of $0$ modulo $n.$ That is,
\[\ord(n):=\begin{cases}\min\{r\geq 1 : n\mid a_{r}\}, & \text{ if it exists}\ \\\infty, &\text{ otherwise.}\\\end{cases}\]
Similarly, we set $\l(n):= \lcm(n,\ord(n))$ whenever $\ord(n)<\infty$.
Lastly, we define 
\[I_{F}(n):=\begin{cases}1/\l(n), & \text{ if } \ord(n)<\infty\ \\0, &\text{ otherwise.}\\\end{cases}\]
\end{definition}

\noindent Next we characterize sets of indices with properties identical to those defined in \cite{sk} or \cite{sa}. For $G(x)\in\mathbb{Z}[x]$, we denote $\mathscr{B}_{F,\,G,\,k}$ to be the set of positive integers $n$ satisfying the following conditions: 
\begin{enumerate}[(1)]
   \item $k\mid \gcd(G(n),a_n),$
   \item If $p\mid \gcd(G(n),a_n)$ for some $p$, then $p\mid k.$
\end{enumerate}
\vspace{0.1cm}
\noindent We denote $\mathscr{A}_{F,\,G,\,k}$ to be the subset of $\mathscr{B}_{F,\,G,\,k}$ defined by $$\mathscr{A}_{F,\,G,\,k}=\{n\ge 1 : \gcd(G(n),a_n)=k\}.$$

\noindent For ease of notations, we let $\mathscr{B}_{F,\,k}=\mathscr{B}_{F,\,x,\,k}$  and $\mathscr{A}_{F,\,k}=\mathscr{A}_{F,\,x,\,k}$ unless stated otherwise, as we consider the case of $G(x)=x$ till Section 5. It can be seen that in the previous work, $\mathscr{B}_{F,\,k}$ was used as an auxiliary tool to find the density of $\mathscr{A}_{F,\,k}$ and did not have any specifications of its own. However, in dynamical sequences, we shall see that this set has an importance of its own.

In this article, we deal with the class of polynomials defined below.
\begin{definition}\label{df1}
A polynomial $F(x)\in \mathbb{Z}[x]$ is said to be $\textsl{dynamically deficient}$ if the sum \begin{equation}\sum_{p\mid a_p}\frac{1} {p}\end{equation} is finite.
\end{definition}We give a detailed overview of $\t$ polynomials in the next section.
\noindent In spirit of the work done in \cite{sk} or \cite{sa}, we investigate the asymptotic densities of the sets $\mathscr{B}_{F,\,k}$ and $\mathscr{A}_{F,\,k}$ for polynomials $F(x)\in\mathbb{Z}[x]$ of degree at least $2$. Our first theorem establishes that these sets have positive density in positive integers if they contain at least one element. The methods used in this paper are somewhat similar to the ones used in \cite{sk}. Two main ideas that we exploit in this paper are the properties of $\textsl{rigid divisibility sequences}$ and the results on the proportion of primes that divide at least one term of the sequence $(a_n)$, which have been illustrated in \cite{be}. These ideas, in conjunction with the already known methods, help us obtain density estimates. Moreover, we are also able to characterize $\mathscr{B}_{F,\,k}$  and $\mathscr{A}_{F,\,k}$ with certain assumptions on the polynomial $F$ (Lemma \ref{lm8}). We now proceed to state our first theorem.

\begin{theorem}\label{thm1}
Let $F(x)\in\mathbb{Z}[x]$ be a polynomial of degree at least 2.
For any fixed positive integer $k,$
\begin{enumerate}[(1)]
   \item The asymptotic density of $\mathscr{B}_{F,\,k}$ exists and is positive if and only if $\mathscr{B}_{F,\,k}\neq \emptyset.$
   \item If $F$ has linear coefficient zero, then $\mathscr{A}_{F,\,k}$ has a positive asymptotic density if and only if $\mathscr{A}_{F,\,k}\neq \emptyset.$
\end{enumerate}
\end{theorem}

\noindent Our second result provides an expression for the asymptotic densities of these sets in terms of $\r(n)$.
\begin{theorem}\label{thm2}
Let $F$ be a $\t$ polynomial of degree at least 2 and $\mu$ be the M\"{o}bius function. Then for any fixed positive integer $k,$
\begin{enumerate}[(1)]
   \item The asymptotic density of $\mathscr{B}_{F,\,k}$ is equal to $$\sum_{\gcd(d,\,k)=1}{\mu(d)}\,{\r(dk)}.$$ 
   \item The asymptotic density of $\mathscr{A}_{F,\,k}$ is equal to $$\sum_{\substack{d\,=\,1}}^{\infty}{\mu(d)}\,{\r(dk)}.$$ 
\end{enumerate}
\end{theorem}
Our last theorem shows that the asymptotic density of the set $\mathscr{A}_{F,\,ax+b\,,1}$ exists and is positive for any $F(x)\in\mathbb{Z}[x]$. The author of \cite{rj} found that the analytic density of the set of primes that divide at least one element of $(a_n)$ is zero for specific families of polynomials. This allows for a strengthening of this theorem for specific $F$, as pointed out at the end of the article (see Remark \ref{rm2}).

\begin{theorem}\label{thm3}
Let $F$ and $G$ be integer polynomials where $F$ has degree at least $2$ and $G$ is a linear polynomial with co-prime coefficients; then the asymptotic density of $\mathscr{A}_{F,\,G\,,1}$ exists and is positive.
\end{theorem}
\subsection*{Notations}
Throughout the article, the letters $p$ and $q$ denote prime numbers and $\mathcal{P}$ denote the set of prime numbers. For any set of integers $\mathcal{S},$ we denote $\mathcal{S}(x)=S\cap [1,x].$
We define $$\mathbf{d}(\mathcal{S})=\lim_{x\to +\infty}\frac{\#\mathcal{S}(x)}{x}$$ if the limit exists. For a subset $\mathcal{G}$  of $\mathcal{P},$ we define $$\mathbf{D}(\mathcal{G})=\lim_{x\to +\infty}\frac{\#\mathcal{G}(x)}{\#\mathcal{P}(x)}.$$
 We employ the Landau-Bachmann notation $\mathcal{O}$ and $o$ as well as their associated Vinogradov notation $\ll$ and $\gg.$

The paper is organized as follows. In Section 2, we discuss the properties and distribution of $\t$ polynomials. In Section 3, we include some preliminary results of different flavors based on the sequence $(a_n)$. In Section 4, we study the structure of sets $\mathscr{A}_{F,\,k}$ and $\mathscr{B}_{F,\,k}$ and also prove Theorem \ref{thm1}.
In Section 5, we give explicit expressions for the asymptotic densities of $\mathscr{A}_{F,\,k}$ and $\mathscr{B}_{F,\,k}$ for $\t$ polynomials $F$, thus proving Theorem \ref{thm2}. 
Finally, in Section 6, we prove Theorem \ref{thm3} and highlight some difficulties we face in proving variants of Theorem \ref{thm1} and \ref{thm2} for $G(x)\neq x.$
\subsection*{Acknowledgments}
I would like to thank Emanuele Tron and Seoyoung Kim for looking at the article and providing valuable comments to improve its quality. I am thankful to Peter Mueller, David Speyer and Will Sawin for their answers on MathOverflow post \cite{mo} and Thomas Tucker for helpful discussions regarding the Proposition \ref{p} of the paper. I am grateful to Ayan Nath for his constant support and helpful advice. I am indebted to the anonymous referee for helpful comments.

\section{Discussion on dynamically deficient polynomials}

In this section, we discuss the distribution and examples of dynamically deficient polynomials. We provide a partial description of these polynomials in Proposition \ref{prop}, which gives us evidence that these polynomials are not rare. This characterization, combined with previous findings, helps identify some examples of $\t$ polynomials. We finish this section by demonstrating that almost all polynomials of degree $d$ are $\t$, supporting our intuition.

We let $\mathscr{S}_{F}$ be the set of primes $p$ for which $\ord(p)=p.$ Using an analysis of the cyclic structure of the polynomial map in $\mathbb{Z}/p\mathbb{Z},$ we get $p\in\mathscr{S}_{F}$ if and only if $F$ is a cyclic permutation of $\mathbb{F}_{p}$. Based on this implication, we present a classification of polynomials that are bijective in $\mathbb{F}_{p}$ for infinitely many $p,$ which in turn is a simplified form of \cite[Theorem 2]{turnwald}.
\begin{proposition}\label{prop}
If $\varphi(x)\in\mathbb{Z}[x]$ is bijective in $\mathbb{F}_{p}$ for infinitely many $p,$ then it is a composition of linear polynomials $cx+d\in \mathbb{Q}[x]$ and Dickson polynomials $D_{q}(x,a)$ with $a\in\mathbb{Z}$ where $q$ is an odd prime and $a=0$ if $q=3.$ (The unique polynomial $D_{n}(x,a)$ with $D_{n}(x+\frac{a}{x},a)=x^{n}+(\frac{a}{x})^{n}$ is called Dickson polynomial of degree $n$ and parameter $a.$)
\end{proposition}

Utilising the ideas above, we give some examples of $\t$ polynomials below.
\begin{enumerate}[(1)]
   \item Let $F(x)\in \mathbb{Z}[x]$ be a polynomial of the form $x^{d}+x^{e}+c$ with $d>e\geq2$. By \cite[Proposition 7 -(4)]{ug}, if $d\not\equiv e \Mod{p-1}$, $p>d$ and $p\nmid c$ then,  $$p\not\in \mathscr{S}_{F} \text{ if } \gcd(d-e,p-1)<\log_{2}{p}.$$ Thus, if $p>\max(3^d,c),$ then $p\not\in\mathscr{S}_{F}.$ Therefore, one can observe that $F$ is $\t,$ since there are finitely many primes belonging to the set $\mathscr{S}_{F},$ implying that the sum $\sum_{p\,\in\,\mathscr{S}_{F}}1/p$ is finite.
   \vspace{0.1cm}
   \item If $\deg F$ is even and $0$ is a wandering point for $F$, then it cannot be clearly expressed in the form stated in Proposition \ref{prop}. Thus, it is a cyclic permutation in $\mathbb{F}_{p}$ for finitely many $p,$ indicating that $F$ is $\t.$
\end{enumerate}
\vspace{0.1cm}
Based on some numerical evidence and arguments in \cite{ug} for the case $F(x)=x^{d}+c,$ it was hypothesised \cite[Hypothesis 6.1]{ug} that when $d=3,$ the probability that a prime $p>2|c|$ satisfies $p\in \mathscr{S}_{F}$ is $1.812/(p-1).$ Thus, we can argue that the expected value of the sum $$\sum_{p\,\in\,\mathscr{S}_{F}}\frac1{p}$$ is equal to $$\sum_{\substack{ p\,<\,2|c|\\p \,\in\,\mathscr{S}_{F}}}\frac1{p}+\sum_{\substack{ p\,>\,2|c|\\p \,\in\,\mathscr{S}_{F}}}\frac{1.812}{p\cdot (p-1)}$$ which is finite. Therefore, one would expect that all polynomials of the form $x^{d}+c$ are $\t.$ Below, we provide an argument stating that almost all polynomials of degree $d\ge 5$ are $\t.$ The proof is essentially based on David Speyer's answer at \cite{mo}.
\begin{proposition}\label{p}

Almost all monic polynomials $F$ of degree $d\geq 5$ are dynamically deficient.\end{proposition}
 \begin{proof}
We denote $G$ as the Galois group of $F(x)-t$ over $\mathbb{C}[t]$ where $t$ and $x$ are assumed to be algebraically independent over $\mathbb{C}$.

We know that almost all monic polynomials of degree $d$ have Galois group $G$ isomorphic to the symmetric group $S_{d}.$ It follows that these polynomials are not expressible as compositions of linear and Dickson polynomials as these compositions have solvable Galois groups \cite[Theorem 3.11-(i),(ii)]{turnwald} while $S_{d}$ is not solvable for $d>4.$ By Proposition \ref{prop}, this immediately leads to the fact that $F$ is bijective in $\mathbb{F}_{p}$ for finitely many $p.$ Thus, we conclude that the sum $\sum_{p\,\in\,\mathscr{S}_{F}}1/p$ is finite.
\end{proof}

\section{Preliminaries}

The following lemma discusses some general properties of dynamical divisibility sequences $(a_r)$.
\begin{lemma}\label{lm3}
Let $n$ and $r$ be two positive integers and let $p$ be a prime such that each of $\ord(n),\ord(r)$ and $\ord(p)$ exists, then
\begin{enumerate}[(1)]
    \item \label{lm3.1} $n\mid a_r$ if and only if $\ord(n)\mid r$.
    \item \label{lm3.2}$\ord(\lcm(n,r))=\lcm(\ord(n),\ord(r))$.
    \item \label{lm3.3}$n\mid \gcd(r,a_r)$ if and only if $\l(n)\mid r.$
    \item \label{lm3.4}$\l(\lcm(n,r))=\lcm(\l(n),\l(r))$.
    
    \item \label{lm3.5}$\l(p)=p\cdot \ord(p)$ if $\ord(p)<p$.
    \item \label{lm3.6}
    $\mathscr{A}_{F,\,k}$ is nonempty if and only if $k=\gcd(\l(k),a_{\,\l(k)}).$
\end{enumerate}
\end{lemma}

\begin{proof}
The first four statements are basic properties of the divisibility sequences, and their proofs are given in \cite[Lemma 2.3]{sk}. For (\ref{lm3.5}), note that if $\ord(p)<p$, then $\gcd(\ord(p),p)=1$ and hence the result.

\noindent The proof of (\ref{lm3.6}) follows directly from arguments in \cite[Proposition 2.4]{sk} and \cite[Lemma 2.2]{sa}. However, we include it here for completeness. Since $n$ divides both $\l(n)$ and $a_{\,\l(n)},$ we know $n\mid \gcd(\l(n),a_{\,\l(n)})$ for all $n.$ Thus if $k\in\mathscr{A}_{F,\,k},$ then $k=\gcd(n,a_n)$ for some $n.$ By \ref{lm3.3}, this indicates that 
$\gcd(\l(k),a_{\,\l(k)})\mid \gcd(n,a_n)=k$ and $k=\gcd(\l(k),a_{\,\l(k)}).$
\end{proof}

\noindent 
When investigating these sequences, a simple question arises: How many primes show up as divisors in the first $N$ terms of $(a_n)?$ To answer this question, we define the set $\mathscr{Q}_{\beta}$  to count these primes and determine the rate at which $\#\mathscr{Q}_{\beta}(x)$ grows.

\noindent
For each $\beta>0$ and a dynamical sequence $(a_n),$ we define the set 
$$\mathscr{Q}_{\beta}=\bigg\{p : \ord(p)\leq \beta\cdot \frac{\log{p}}{\log{d}}\bigg\}$$ where $d=\deg F$.
We present the dynamical analogue of a folklore argument implicit in \cite[Section 3, p. 212]{ff} and \cite[Lemma 2.6]{sk} to bound the cardinality of $\mathscr{Q}_{\beta}(x)$. The estimates obtained here are similar to \cite[Theorem 1.1i]{ca}.

\begin{lemma}\label{lm4}
For each $\beta>0$, there exists a constant $C$ such that $\#\mathscr{Q}_{\beta}(x)< C\, x^{\beta}$ for all $x\ge 1.$
\end{lemma}

\begin{proof}
From the definition of $\mathscr{Q}_{\beta}(x),$ we observe that $$\exp{(\#\mathscr{Q}_{\beta}(x))}\le \,2\cdot \prod_{p\,\in\,\mathscr{Q}_{\beta}(x)}p.$$
Therefore, we have that \begin{equation}\label{e5}\#\mathscr{Q}_{\beta}(x)\le \,\log{2}+\sum_{p\,\in\,\mathscr{Q}_{\beta}(x)}\log{p}.\end{equation}
Moreover, $\ord(p)\le \,\beta\cdot {\log{p}}/{\log{d}}\,\le \,\beta\cdot {\log{x}}/{\log{d}}\,$ for each prime $p\in \mathscr{Q}_{\beta}(x),$ which shows that $$p\mid a_{\,\ord(p)},\text{ whenever } \ord(p)\le \,\beta\cdot {\log{x}}/{\log{d}}.$$ 

\vspace{0.3cm}
\noindent This, in turn, leads to following divisibility $$\prod_{p\,\in\,\mathscr{Q}_{\beta}(x)}p \,\mid \prod_{n\,\le\,\beta\cdot {\log{x}}/{\log{d}}}a_{n},$$ which along with (\ref{e5}) reduces to $$\#\mathscr{Q}_{\beta}(x)\le \log{2}+\sum_{n\,\le\,\beta\cdot {\log{x}}/{\log{d}}}\log{|a_n|}.$$
 Also, by \cite[Theorem 1]{poly}, if we set  $$\alpha=\lim_{n\to +\infty}a_n^{1/d^{n}},$$ then $\alpha$ is either an integer or irrational number greater than 1. Furthermore, based on a classical technique for analysis of polynomial recursions in \cite[Section 2]{poly}, we get that $a_n=A\alpha^{d^n}+B+\OO(\alpha^{-d^n})$ where $A$ and $B$ are algebraic numbers depending only on $F$.
 As a result, we conclude that $$\#\mathscr{Q}_{\beta}(x)\ll \sum_{n\,\le\,\beta\cdot {\log{x}}/{\log{d}}}d^{n}\ll d^{\,\beta\cdot\, \frac{\log{x}}{\log{d}}}=x^{\beta},$$ as desired.\end{proof}

\begin{remark}
It can be seen that the bounds for the rank of apparition of primes in the case of Lucas sequences and elliptic divisibility sequences have been proven in \cite[Lemma 2.4]{sa} and \cite[Lemma 2.6]{sk} respectively. 
Let $(D_n)_{n\ge 1}$ be an elliptic divisibility sequence and let $r_n$ be the rank of apparition for $n$ for the elliptic divisibility sequence, that is, $$r_n=\min\{r\ge 1 : n\mid D_r\}.$$ Then for all $x,\beta>0$, it is shown in \cite[Lemma 2.6]{sk} that $$\#\{p\le x : r_p\le p^{\beta}\}\ll x^{3\beta}.$$  

However, current techniques seem insufficient to prove such bounds in the dynamical case as the growth rate of the sequence $(a_n)$ is quite large as compared to that of linear recurrences or elliptic divisibility sequences. As a result, we get a logarithmic factor in our bounds.

It has been conjectured based on some numerical experiments in \cite[Conjecture 18.3]{po} and \cite[Conjecture 14]{silver} that the bounds obtained for Lucas sequences or elliptic divisibility sequences hold for their dynamical cousins as well.
\end{remark}Our next lemma is a dynamical analogue of sums in \cite{romanoff} and a somewhat stronger analogue of \cite[Theorem 11]{silver}. Although, the crucial idea is same.
\begin{lemma}\label{lm5}
Let $\e,\beta>0$ and $\epsilon\geq 0$ be three constants such that $\delta=\e-\epsilon>0$ and $\beta+\epsilon<1.$ Then, 
$$\sum_{p\,>\,z}\frac{(\log{p})^{\epsilon}}
{p\cdot \ord(p)^{\e}}\ll \frac{1}{(\log{z})^{\delta}}.$$
\end{lemma}

\begin{proof}
We divide our summation into two parts;
$$\sum_{p>z}\frac{(\log{p})^{\epsilon}}{p\cdot \mathfrak{o}(p)^{\e}}=S_{1}+S_{2}$$ where $S_{1}$ is the summation over the primes not in $\mathscr{Q}_{\beta}$ and $S_{2}$ is the summation over the rest of primes.
For $S_{1}$, by the definition of $\mathscr{Q}_{\beta}$ and the fact that $n$-{th} prime $p_{n}$ is at least $n\cdot \log{n}$ for $n\geq 2$ \cite{ro}, we obtain the following,
\begin{align*}\sum_{\substack{p\,\ge\, z\\p\,\notin\, \mathscr{Q}_{\beta}}}\frac{(\log{p})^{\epsilon}}{p \cdot \ord(p)^{\e}}&\ll \sum_{z\,\le\, p\,\le\, p_{z}}\frac{1}{p \cdot (\log{p})^{\delta}}+\sum_{p\,\ge\, p_{z}}\frac{1}{p \cdot (\log{p})^{\delta}}\\&\ll \frac{1}{(\log{z})^{\delta}}+\sum_{x\,\ge\, z}\frac{1}{x\cdot (\log{x})^{1+\delta}}\\&\ll \frac1{(\log z)^{\delta}}+\int_z^\infty \frac{\mathrm{d}t}{t(\log t)^{1+\delta}}\\&\ll\frac{1}{(\log{z})^{\delta}}+\frac{1}{(\delta)\cdot (\log{z})^{\delta}}\ll \frac{1}{(\log{z})^{\delta}}.
\end{align*}\vspace{0.1cm}

\noindent For $S_{2}$, it follows from partial summation that 
\begin{align*}\noindent\sum_{\substack{p\,\ge\, z\\p\,\in\, \mathscr{Q}_{\beta}}}\frac{(\log{p})^{\epsilon}}{p\cdot \ord(p)^{\e}}&\leq \sum_{\substack{p\,\ge\, z\\p\,\in\, \mathscr{Q}_{\beta}}}\frac{1}{p^{1-\epsilon}\cdot \ord(p)^{\e}}\\& \le \sum_{\substack{p\,\ge\, z\\p\,\in \,\mathscr{Q}_{\beta}}} \frac{1}{p^{1-\epsilon}}=\frac{\#\mathscr{Q}_{\beta}(t)}{t^{1-\epsilon}}\bigg|^{\infty}_{t=z} + (1-\epsilon) \int_{z}^{\infty} \frac{\#\mathscr{Q}_{\beta}(t)}{t^{2-\epsilon}}\mathrm{d}t\\&\ll \frac{1}{z^{1-\epsilon-\beta}}.
\end{align*}
Thus, we obtain$$\sum_{p\,>\,z}\frac{(\log{p})^{\epsilon}}{p\cdot \mathfrak{o}(p)^{\e}}\ll \frac{1}{(\log{z})^{\delta}}+\frac{1}{z^{1-\epsilon-\beta}}\ll \frac{1}{(\log{z})^{\delta}},$$ as desired.\end{proof}

\noindent For a given set $\mathscr{L}$ of integers, the set of its non-multiples is defined as \begin{equation}\mathcal{N}(\mathscr{L})=\{n\geq 1 : s\nmid n \text{ for all } s\in \mathscr{L}\}\end{equation}We shall be using the following lemma regarding such sets.

\begin{lemma}\label{lm6}
If $\mathscr{L}$ is a set of positive integers such that 
$$\sum_{s\,\in\, \mathscr{L}}\frac{1}{s}<\infty,$$ then $\mathcal{N}(\mathscr{L})$ has an asymptotic density. Moreover, if $1\notin \mathscr{L}$, then $\mathcal{N}(\mathscr{L})$ has positive asymptotic density.
\end{lemma}

\begin{proof}
See \cite[Lemma 2.3]{sa}.
\end{proof}

Next, we discuss a particular class of divisibility sequences known as $\textsl{rigid divisibility sequences}$.
An integer sequence $(b_n)$ is a $\textsl{rigid}$ $\textsl{divisibility}$ $\textsl{sequence}$ if for every prime $p$, the following two properties hold:
\begin{enumerate}[(a)]
   \item If $\nu_{p}(b_{n})>0$, then $\nu_{p}(b_{nk})=\nu_{p}(b_{n})$ for all $k\geq 1$, and 
   \item If $\nu_{p}(b_{n})>0$ and $\nu_{p}(b_{m})>0$, then $\nu_{p}(b_{n})=\nu_{p}(b_{m})=\nu_{p}(b_{\gcd(n,m)}).$
\end{enumerate}

\begin{lemma}\label{lem2}
Let $\varphi(x)\in \mathbb{Z}[x]$ be a  non-constant polynomial whose linear coefficient is zero, then the sequence $(\varphi^{n}(0))_{n\ge 0}$ is a rigid divisibility sequence.
\end{lemma}

\begin{proof}
See \cite[Proposition 3.1, 3.2]{br}.
\end{proof}

The following result gives information on the analytic density of primes $p$ that do not divide any term of the sequence $(a_n)$. We need a simplified variant of \cite[Theorem 3.1]{be} for our purposes.
\begin{lemma}\label{lm14} Let $\varphi(x)$ be a polynomial with integral coefficients and degree at least $2$. Let $\mathcal{A}$ and $\mathcal{T}$ be finite subsets of $\mathbb{Z}$ such that the forward orbit $$\mathcal{O}_{{\varphi}(z)}= \{z, \varphi(z), \varphi^{2}(z)\ldots \}$$ is finite for all $z\in \mathcal{A}$ with a possible exception of at most one integer and infinite for all $z\in \mathcal{T}$. Then, there exists a positive integer $M$ and a positive analytic density of primes $p$ such that for any $\gamma \in \mathcal{T}$, any $\alpha \in \mathcal{A}$, any $p\in \mathcal{P}$, and any $m\ge M$,  $$\varphi^{m}(\gamma)\not\equiv \alpha \Mod{p}.$$ 
\end{lemma}
\begin{proof}
See \cite[Theorem 3.1]{be}.
\end{proof}

Our next lemma gives a modified version of Mertens' formula \cite[Chapter I.1, Theorem 11]{gt} for primes that divide at least one term of the sequence $(a_n).$

\begin{lemma}\label{lm1}
There exist constants $B_0$ and $\gamma \in (0,1]$ such that for $B\ge B_0$, we have,
$$\prod_{\substack{q\,\le\, B\\ \ord(q)\,<\,\infty }}{\left(1-\frac{1}{q}\right)}\gg \frac{1}{(\log{B})^{1-\gamma}}.$$
\end{lemma}
\begin{proof}
Let $\varphi(x)=F(x)$ and $\mathcal{T}=\mathcal{A}=\{0\}$, then we know that $a_n=\varphi^{n}(0)$ is unbounded. Therefore, we can apply Lemma~\ref{lm14} to these sets to obtain a positive analytic density 
 of primes $p$ such that $a_n\not\equiv 0 \pmod{p}$ for all $n\ge M$ where $M$ is an absolute constant.

\noindent Let $\mathcal{L}$ be the set of primes $p$ such that $\ord(p)=\infty$. Then by the arguments above, we have that $\mathbf{D}(\mathcal{L})>0$. First, we handle the case when $\mathbf{D}(\mathcal{L})<1$ and so we let $\mathbf{D}(\mathcal{L})=\gamma$ for $\gamma \in (0,1).$
Therefore,
$$\sum_{\substack{q\,\le\, B\\ q\,\in\,\mathcal{P}\setminus\mathcal{L} }}{\log{\left(1-\frac{1}{q}\right)}}= -\sum_{\substack{q\,\le\, B\\ q\,\in\,\mathcal{P}\setminus\mathcal{L} }}\frac{1}{q}+\OO(1).$$
We know from elementary calculations that $$\sum_{\substack{q\,\le\, B\\ q\,\in\,\mathcal{P}\setminus\mathcal{L} }}\frac{1}{q}=(1-\gamma)\cdot \log{\log{B}}+\OO(1).$$ This gives us a bound $$\sum_{\substack{q\,\le\, B\\ q\,\in\,\mathcal{P}\setminus\mathcal{L} }}{\log{\left(1-\frac{1}{q}\right)}}=-(1-\gamma)\cdot \log{\log{B}}+\OO(1)$$ from which, we get the inequality $$\prod_{\substack{q\,\leq\, B\\ q\,\in\,\mathcal{P}\setminus\mathcal{L} }}{\left(1-\frac{1}{q}\right)}\gg \frac{1}{(\log{B})^{1-\gamma}}.$$
Now, note that in case $\mathbf{D}(\mathcal{L})=1,$ repeating the same arguments as before, we get that $$\sum_{\substack{q\,\le\, B\\ q\,\in\,\mathcal{P}\setminus\mathcal{L} }}{\log{\left(1-\frac{1}{q}\right)}}=-o(\log{\log{B}})+\OO(1),$$ which leads to $$\prod_{\substack{q\,\le\, B\\ q\,\in\,\mathcal{P}\setminus\mathcal{L} }}{\left(1-\frac{1}{q}\right)}\gg \frac{1}{(\log{B})^{o(1)}}\gg\frac1{(\log{B})^{1-\gamma}},$$ for any fixed $\gamma \in (0,1]$ provided $B\ge B_{0}$ for some constant $B_{0}$. \end{proof}

\section{Discussing the asymptotic density of $\mathscr{A}_{F,\,k}$ and $\mathscr{B}_{F,\,k}$}

In this section, we give structural characterizations of $\mathscr{B}_{F,\,k}$ and $\mathscr{A}_{F,\,k}$ as scaled sets of non-multiples. For all $F$ with $0$ as a wandering point, we can completely describe the structure of $\mathscr{B}_{F,\,k}$. However, while dealing with $\mathscr{A}_{F,\,k},$ we further assume that $F$ has linear coefficient as zero. This is assumed as the sequence $(a_n)$ defined for such $F$ has $\text{rigid divisibility}$ properties, giving us a better grasp on the $p$-adic properties of these sequences. We show that these sets have positive asymptotic densities if and only if they are not empty (Theorem \ref{thm1}), based on this characterization. Now, we state the lemma analyzing the structure of these sets.

\begin{lemma}\label{lm8}
Let $k$ be a positive integer such that $\ord(k)<\infty$. We define
\begin{equation}\label{eq1}\mathscr{T}_{k}= \bigg\{\frac{\l(kp)}{\l(k)} : \ord(p)<\infty\text{ and }p\nmid k\bigg\}\end{equation} and \begin{equation}\label{eq10}\mathscr{L}_{k}=\big\{p : p\mid k \text{ and } \nu_p(\l(k))<\nu_p(a_{\,\l(k)})\big\} \cup \mathscr{T}_{k}.\end{equation}
If $\mathscr{B}_{F,\,k}$ is nonempty, then we have that 
\begin{equation}\mathscr{B}_{F,\,k}=\{\l(k) m : m\in \mathcal{N}(\mathscr{T}_{k})\}.\end{equation}

\noindent Moreover, if $F$ has linear coefficient zero  and  $\mathscr{A}_{F,\,k}$ is nonempty for a positive integer $k$, we have,
\begin{equation}\mathscr{A}_{F,\,k}=\{\l(k)m : m\in \mathcal{N}(\mathscr{L}_{k})\}.\end{equation}
\end{lemma}
\begin{proof}
Firstly, we give a characterization of $\mathscr{B}_{F,\,k}$. If $n\in\mathscr{B}_{F,\,k},$ then $k\mid a_n$ and $\l(k)\mid n$ by Lemma ~\ref{lm3}-(\ref{lm3.3}). Hence, it is easy to see that for a positive integer $m$, $\,\l(k) m\in \mathscr{B}_{F,\,k}\,$ if and only if 
$p\nmid \gcd(\l(k)m,a_{\,\l(k)\,m})$ for every prime $p$ such that $\ord(p)<\infty$ and $p\nmid k$. This further implies that $\l(p)\nmid \l(k) m$ for such $p$ by Lemma~\ref{lm3}-(\ref{lm3.3}) which in turn is equivalent to  $$\frac{\lcm(\l(k),\l(p))}{\l(k)}= \frac{\l(kp)}{\l(k)}\nmid m,$$ due to Lemma ~\ref{lm3}-(\ref{lm3.4}).

\noindent Therefore, we conclude that 
$$\mathscr{B}_{F,\,k}=\{\l(k) m : m\in \mathcal{N}(\mathscr{T}_{k})\}.$$

For the second part, note that $\l(k) m \in \mathscr{A}_{F,\,k}$  for some $m$ if and only if $\nu_{p}(\gcd(\l(k) m,a_{\,\l(k)\, m}))=\nu_{p}(k)$ for all primes $p.$
For the primes $p\nmid k$ such that $\ord(p)<\infty,$ we must have 
$p\nmid \gcd(\l(k) m,a_{\,\l(k)\, m}),$ which is equivalent to \begin{equation}\label{e1}\frac{\l(kp)}{\l(k)}\nmid m,\end{equation} due to above argument.

Lastly, we are left considering the case of primes $p\mid k.$ Due to Lemma \ref{lem2}, we have that $\nu_{p}(a_{\,\ord(p)\cdot r})=\nu_{p}(a_{\,\ord(p)})$ for all positive integers $r$ and as $\mathscr{A}_{F,\,k}$ is nonempty, we know that $$\nu_{p}(\gcd(\l(k),a_{\,\l(k)}))=\nu_{p}(k),$$ by Lemma ~\ref{lm3}-(\ref{lm3.6}). First, we handle primes $p$ such that $\nu_{p}(\l(k))\ge \nu_{p}(a_{\,\l(k)})$ for which $\nu_p(a_{\,\l(k)})=\nu_p(k)$. We have, $$\nu_{p}(\gcd(\l(k) m,a_{\,\l(k)\, m}))=\nu_{p}(a_{\,\l(k)\, m})=\nu_{p}(a_{\,\l(k)})=\nu_{p}(k),$$ as desired. Next, we consider the case when 
$\nu_{p}(\l(k))< \nu_{p}(a_{\,\l(k)})$ for which $\nu_{p}(\l(k))=\nu_{p}(k).$ Thus, 
$$\nu_{p}(\gcd(\l(k) m,a_{\,\l(k)\, m}))=\min(\nu_{p}(\l(k) m),\nu_{p}(a_{\,\l(k)\, m})),$$ which is greater than $\nu_{p}(k)$ if and only if $p \mid m.$  Therefore, 
\begin{equation}\label{e2}\nu_{p}(\gcd(\l(k) m,a_{\,\l(k)\, m}))=\nu_{p}(k) \text{ if and only if }p\nmid m\end{equation} and hence, by (\ref{e1}) and (\ref{e2}) it follows that $\l(k) m \in \mathscr{A}_{F,\,k}$ if and only if $m\in \mathcal{N}(\mathscr{L}_{k})$. Thus, we deduce that $$\mathscr{A}_{F,\,k}=\{\l(k) m : m\in \mathcal{N}(\mathscr{L}_{k})\},$$ as desired. 
\end{proof}

\begin{definition}
A prime $p$ is said to be $\textsl{anomalous}$ if $\ord(p)=p$ and $\textsl{non-anomalous}$ if $\ord(p)\neq p$ and $\ord(p)<\infty.$
\end{definition}

\noindent Note that the definition of anomalous primes is  similar to the one in \cite[Definition 3.2]{sk}. However, in the case of elliptic divisibility sequences, contrary to their dynamical counterparts, many results about distribution and properties of $\text{anomalous}$ primes have been obtained effectively. Authors in \cite{cyclic1} have studied and experimented with the distribution of $\text{anomalous}$ primes in case of dynamical divisibility sequences.

We are now in a position to prove Theorem \ref{thm1}. 
\begin{proof}[Proof of Theorem \ref{thm1}]
We will be proving Theorem \ref{thm1}-(1) as Theorem \ref{thm1}-(2) follows in a similar way. Observe that as $\mathscr{B}_{F,\,k}$ is nonempty, we must have $\ord(k)<\infty$. By Lemma ~\ref{lm6} and Lemma~\ref{lm8}, if we prove that $$\sum_{n \,\in\, \mathscr{T}_k}\frac{1}{n}$$ converges, then, as $1\notin \mathscr{T}_{k},$ we get that  $\mathbf{d}(\mathscr{B}_{F,\,k})>0$.
Now using (\ref{eq1}), we obtain that
\vspace{0.1cm}
\begin{align*}\sum_{n\,\in\, \mathscr{T}_k}\frac{1}{n}&\ll \sum_{\ord(p)\,<\,\infty}\frac{1}{\l(kp)}\leq \sum_{\ord(p)\,<\,\infty}\frac{1}{\l(p)}\\&=\sum_{\substack{p \text{ is non-}\\\text{anomalous}}}\frac{1}{p\cdot \ord(p)}+ \sum_{\substack{p \text{ is }\\\text{anomalous }}}\frac{1}{p},\end{align*} where the convergence of the first sum follows by taking $\epsilon=0,\,\e=1\text{ and }\beta<1$ in Lemma \ref{lm5} and the second sum
converges as $F$ is $\t$. Hence our proof is complete.
\end{proof}

\section{Calculating the explicit  densities of $\mathscr{A}_{F,\,k}$ and $\mathscr{B}_{F,\,k}$}
This section constitutes the main portion of the article. We start by demonstrating convergence of a sum involving $\r(n),$ analogous to the sums in \cite[Lemma A.1]{sk} and \cite[Lemma 3.2]{sa}.

In case of Fibonacci numbers, it is well-known that $5$ is the only $\text{anomalous}$ prime, while it is known that the sum of reciprocal of $\text{anomalous}$ primes converges for elliptic divisibility sequences arising from different elliptic curves. In our case, we restricted $F$ to be dynamically deficient so as to bound the sum $\sum_{\ord(p)=p}1/p$ by a constant. We find asymptotic densities of $\mathscr{A}_{F,\,k}$ and $\mathscr{B}_{F,\,k}$ using this assumption and the sum in Lemma \ref{lm9}. However, as \cite[Remark 4.2]{sa} points out, there is no known means of establishing $\textsl{a priori}$ the requirement for the density to be zero or not, or even just showing non-negativity, without going through the related characterization as in Lemma \ref{lm8}. Now, we state our lemma.
\begin{lemma}\label{lm9}
Let $\mu$ be the M\"{o}bius function.
The sum $$\sum_{\substack{n\,=\,1}}^{\infty}{|\mu(n)|}\,{\r(n)}$$ is finite.
\end{lemma}
\begin{proof}
The proof is similar to \cite[Lemma A.1]{sk}. Since $F$ is dynamically deficient, we can choose a constant $B\ge B_0$ ($B_0$ is the same constant as in the statement of Lemma~\ref{lm1}) such that \begin{equation}\label{eq2}\sum_{\substack{p\,>\,B\\p\text{ is anomalous}}}\frac{1}{p}\leq \frac{1}{2}.\end{equation}
For rest of the argument, denote $P(n)$ to be the greatest prime factor of $n.$  Since $\l(n)\ge n,$ we obtain
$$\sum_{\substack{P(n)\,\le\,B}}^{\infty}{|\mu(n)|}\,{\r(n)}\le \sum_{P(n)\,\le\,B}^{\infty}\frac{1}{n}\leq \prod_{p\,\leq\, B}\left(1+\frac{1}{p}\right)<\infty.$$
\vspace{0.3cm}
Now, we suppose $P(n)>B$ for the rest of the proof. Hence,
\begin{equation}\label{eq3}\sum_{\substack{P(n)>B}}^{\infty}{|\mu(n)|}\,{\r(n)}= \sum_{\substack{P(n)\,>\,B\\ P(n) \text{ is non-
}\\\text{anomalous} }}^{\infty}{|\mu(n)|}\,{\r(n)}+\sum_{\substack{P(n)\,>\,B\\ P(n) \text{ is }\\\text{anomalous} }}^{\infty}{|\mu(n)|}\,{\r(n)}.\end{equation}

\noindent Firstly, we deal with the first sum in (\ref{eq3}) involving primes which are non-anomalous. From now on, we assume that the sum $\sideset{}{'}\sum$ runs over indices $n$ such that $\ord(n)<\infty$, $P(n)>B$ and $P(n)$ is non-anomalous. Observe that for a positive integer $n$ such that $\ord(n)<\infty$, $$ \lcm(n,\ord(P(n)))\mid \lcm(n,\ord(n))=\l(n)$$ by Lemma~\ref{lm3}-(\ref{lm3.2}). Therefore,
\begin{align*}\sum_{\substack{P(n)\,>\,B\\ P(n) \text{ is non-}\\\text{anomalous} }}^{\infty}{|\mu(n)|}\,{\r(n)}\,\le\, \sideset{}{'}\sum_{n\,=\,1}^{\infty}\,\frac{|\mu(n)|}{\l(n)}\,\le\, \sideset{}{'}\sum_{n\,=\,1}^{\infty}\frac{1}{\lcm(n,\ord(P(n))}.\end{align*} Let $p=P(n).$ We can write 
$$\lcm(n,\ord(p))=\l(p)\cdot m$$ where $P(m)\leq p$ and $\ord(m)<\infty.$ Also, if $p$ and $\lcm(n,\ord(p))$ are known, then $n$ can be chosen in at most $\tau(\ord(p))$ ways. Henceforth,
$$ \sideset{}{'}\sum_{n\,=\,1}^{\infty}\frac{1}{\lcm(n, \ord(P(n))}
\ll \sum_{\substack{p\,>\,B\\p \text{ is non-}\\\text{anomalous}}}\frac{\tau(\ord(p))}{\l(p)}\left(\sum_{\substack{P(m)\,\leq\, p\\\ord(m)\,<\,\infty}}\frac{1}{m}\right)
$$ where $\tau(n)$ denotes the number of divisors of $n.$
Applying Lemma \ref{lm1}, we have,
$$\sum_{\substack{P(m)\,\leq\, p\\\ord(m)\,<\,\infty}}\frac{1}{m}\leq \prod_{\substack{q\,\leq\, p\\ \ord(q)\,<\,\infty }}{\left(1-\frac{1}{q}\right)^{-1}} \ll (\log{p})^{1-\gamma} 
$$
for all prime numbers $p>B$ and $\gamma \in (0,1].$ We deduce that 
\begin{equation*}
\sum_{\substack{p\,>\,B\\p \text{ is non-}\\\text{anomalous}}}\frac{\tau(\ord(p))}{\l(p)}\left(\sum_{\substack{P(m)\,\leq\, p\\\ord(m)\,<\,\infty}}\frac{1}{m}\right)\leq \sum_{ \substack{p\,>\,B\\p \text{ is non-}\\\text{anomalous}}}\frac{\tau(\ord(p))\cdot (\log{p})^{1-\gamma}}{\l(p)}.  
\end{equation*}
Moreover, we know that $\tau(n)\ll n^{\epsilon}$ for any $\epsilon>0$ from \cite[Chapter I.5, Corollary 1.1]{gt} and $\l(p)=p\, \ord(p) $ (See Lemma~\ref{lm3}-(\ref{lm3.5})) since $p$ is non-anomalous. As a consequence,
\begin{align*} \sum_{ \substack{p\,>\,B\\p \text{ is non-}\\\text{anomalous}}}\frac{\tau(\ord(p))\cdot (\log{p})^{1-\gamma}}{\l(p)} \ll\sum_{p}\frac{(\log{p})^{1-\gamma}}{p\cdot \ord(p)^{1-\epsilon}}.\end{align*}
Our main task is to show convergence of the last sum above. We choose $\epsilon={\gamma}/{100}$ and $\beta<1-{\gamma}$. Then, we apply Lemma \ref{lm4} and Lemma \ref{lm5} with these constants to get that  \[\sum_{p}\frac{(\log{p})^{1-\gamma}}{p\cdot \ord(p)^{1-\epsilon}}\] converges.
\vspace{0.5cm}

\noindent Now, we will be dealing with the second sum in (\ref{eq3}) where $P(n)$ is anomalous. Again, let $p=P(n).$  We can assume that $n$ is square-free since $\mu(n)$ is non-zero only at square-free values. It follows that $n=p\cdot b$ where $b$ is square-free and $\gcd(b,p)=1$. Therefore,
$$\lcm(p,b,\ord(b))=\lcm(n,\ord(n))=\l(n)$$
and as $$\ord(b)=\lcm(\ord(p_{1}),\ord(p_{2}),\ldots \ord(p_{r}))\text{ where }b=p_{1}\cdot p_{2}\cdots p_{r},$$ we can conclude that $\gcd(p,\ord(b))=1$ implying $\l(n)=p\cdot \l(b).$ So we write \begin{equation}\label{eq4}\sum_{\substack{n\,=\,1}}^{x}{|\mu(n)|}\,{\r(n)}=\sum_{\substack{ P(n)>B\\ P(n) \text{ is }\\\text{ anomalous}}}^{x}{|\mu(n)|}\,{\r(n)}+K,\end{equation} where constant $K$ arises due to the convergent sum in case when $P(n)$ is non-anomalous or $P(n)\le B$. Recall that due to (\ref{eq2}), 
\begin{align*}
\sum_{\substack{ P(n)>B\\ P(n) \text{ is }\\\text{ anomalous}}}^{x}{|\mu(n)|}\,{\r(n)}&=\sum_{\substack{ P
(n)>B \\\ord(n)\,<\,\infty \\P(n) \text{ is anomalous}}}^{x}\frac{|\mu(n)|}{\l(n)}\\&\leq \sum_{\substack{B<p<x\\p\text{ is anomalous}}}\frac{1}{p}\cdot \sum_{\substack{b\,\le\,x\\\ord(b)\,<\,\infty}}\frac{|\mu(b)|}{\l(b)}\\&\leq \frac{1}{2}\cdot \sum_{\substack{b\,\le\,x}}{|\mu(b)|}\,{\r(b)}.\end{align*} Consequently, one can see that applying previous steps repeatedly on the last sum, the largest prime factor gets reduced in each such step and thus, from (\ref{eq4}), we have that $$\sum_{\substack{n\,=\,1}}^{x}{|\mu(n)|}\,{\r(n)}\le 2K.$$ Thus, if we let $x\rightarrow\infty$, we are done.\end{proof}
Now we proceed to prove Theorem \ref{thm2}.
\begin{proof}[Proof of Theorem \ref{thm2}]
The proof is similar to \cite[Theorem 1.2]{sk} and \cite[Theorem 1.4]{sa}.
\noindent For all positive integers $n$ and $d$, we define \[\varrho(n,d)=\begin{cases} 1, & \text{if}\ d\mid a_{n}\\ 0, & \text{otherwise.}\\ \end{cases}\]
Observe that $$\varrho(n,de)=\varrho(n,d)\cdot \varrho(n,e)$$ for all relatively prime positive integers $d$ and $e$ and positive integers $n$.

One can infer that $n \in \mathscr{B}_{F,\,k}$ if and only if $\ord(k)<\infty$, $\l(k)\mid n$ and $\varrho(n,p)=0$ for all prime numbers $p$ such that $p\mid n$ but $p\nmid k$. Henceforth,
\begin{align}\label{eq5}\#\mathscr{B}_{F,\,k}(x)&=\sum_{\substack{n\,\leq\, x\\\l(k)\,\mid\,n}}\prod_{\substack{p\,\mid\,n\\p\,\nmid\, k}}\left(1-\varrho(n,p)\right)=\sum_{\substack{n\,\leq\, x\\\l(k)\,\mid\, n}}\prod_{\substack{d\,\mid\, n\\\gcd(d,\,k)=1}}\mu(d)\cdot \varrho(n,d)\\&=\sum_{\substack{d\,\leq\, x\\\gcd(d,\,k)=1}}\mu(d) \sum_{\substack{m\,\leq\, {x}/{d}\\\l(k)\,\mid\,dm}}\varrho(dm,d).
\end{align}

\noindent Now note that if $\varrho(dm,d)=1$ and $\l(k)\mid dm$, it follows that $\ord(d)<\infty$ and $$\lcm(\ord(d),\l(k))\mid dm.$$ If $\gcd(d,k)=1,$ then this is equivalent to the fact that
$$\frac{\lcm(d,\lcm(\ord(d),\l(k)))}{d}=\frac{\lcm(\l(d),\l(k))}{d}=\frac{\l(dk)}{d}$$ divides $m.$
Therefore,
$$\sum_{\substack{m\,\leq\, {x}/{d}\\\l(k)\,\mid\, dm}}\varrho(dm,d)=\sum_{\substack{m\,\leq \,{x}/{d}\\{\l(dk)}/{d}\,\mid\, m}}1=\left\lfloor\frac{x}{\l(dk)}\right\rfloor=\left\lfloor x\,I_{F}(dk)\right\rfloor$$
which combined with (\ref{eq5}), implies 
$$\#\mathscr{B}_{F,k}(x)=\sum_{\substack{d\,\leq\, x\\\
\gcd(d,\,k)=1}}\mu(d)\left\lfloor x\,I_{F}(dk)\right\rfloor.$$
Now expressing the floor function in terms of fractional parts we get
\begin{equation}\label{eq6}
\#\mathscr{B}_{F,\,k}(x)=x\sum_{\substack{d\,\leq\, x\\
\gcd(d,\,k)=1}}{\mu(d)}\,{I_{F}(dk)}-\sum_{\substack{d\,\leq\, x\\
\gcd(d,\,k)=1}}\mu(d)\{{x}\,{I_{F}(dk)}\}. 
\end{equation} 
By Lemma \ref{lm9} and the fact that $\l(dk)\ge \l(d)$ when $\ord(d)<\infty$, we deduce that $$\sum_{\substack{d\,\leq\, x\\
\gcd(d,\,k)=1}}{|\mu(d)|}\,{\r(dk)}\leq \sum_{\substack{d\,=\,1}}^{\infty}{|\mu(d)|}\,{\r(d)}<\infty.$$
Also, we can see that
\begin{align*}\sum_{\substack{d\,\leq\, x\\
\gcd(d,\,k)=1}}|\mu(d)|\,\{{x}\,{\r(dk)}\}&= \OO(x^{1/2})+\sum_{\substack{x^{1/2}<d\leq x}}|\mu(d)|\,\{{x}\,{\r(dk)}\}\\&\le \OO(x^{1/2})+x\sum_{\substack{d\geq x^{1/2}}}{|\mu(d)|}\,{\r(d)}=o(x)\end{align*} since by Lemma \ref{lm9}, the last series is the tail of a convergent series and hence tends to zero as $x\to \infty. $ Thus, from (\ref{eq6}) we have that \begin{equation}\label{eq7}\frac{\#\mathscr{B}_{F,k}(x)}{x}\rightarrow\sum_{\substack{d\,\leq\, x\\
\gcd(d,\,k)=1}}{\mu(d)}\,{\r(dk)}.\end{equation}
Hence the first part of Theorem \ref{thm2} is proven.

For the second part, by the application of principle of inclusion and exclusion, one sees that 
$$\#\mathscr{A}_{F,\,k}(x)=\sum_{d\,\mid\, k}\mu(d)\#\mathscr{B}_{F,\,dk}(x),$$ 
which on applying (\ref{eq7}) reduces to 
\begin{align*}
\mathbf{d}(\mathscr{A}_{F,\,k})&=\sum_{d\,\mid\, k}\mu(d)\mathbf{d}(\mathscr{B}_{F,\,dk})=\sum_{d\,\mid\, k}\mu(d)\sum_{\substack{\gcd(c,\,dk)=1}}{\mu(c)\,
}{\r(cdk)}\\
&=\sum_{d\,\mid\, k}\sum_{\substack{\gcd(c,\,dk)=1}}{\mu(cd)}\,{\r(cdk)}=\sum_{\substack{t\,=\,1}}^{\infty}{\mu(t)}\,{\r(tk)}.
\end{align*}
since every square-free integer $t$ can be written in a unique way as $t=c\cdot d$, where $c$ and $d$ are square-free integers such that $c\mid k$ and $\gcd(d,k)=1$. Furthermore, note that the rearrangement of sum could be made possible due to the absolute convergence of sum in the lemma \ref{lm9}.
\end{proof}

\section{On the density of $\mathscr{A}_{F,\,ax+b,\,1}$ for general polynomials $F$}

After thoroughly analyzing the sets $\mathscr{A}_{F,\,x,\,k}$, we now turn our attention to their generalizations. One possible generalization is to explore a broader class of polynomials as $G$.

As stated in \cite{cd}, the only conceivable generalization in this regard is for $G$ with all rational roots and no fixed divisors. However, it appears that moving beyond the linear case is possible only for a specific class of $F$, which we discuss at the end of the section (\ref{rm2}). The benefit of these specific $F$ is that, under these circumstances, the set of primes $p$ for which $\ord(p)<\infty$ is zero, which aids in our calculations. In this section, we look at the case where $G(x)=ax+b$ with $\gcd(a,b)=1$ and $k=1$.

In the case of linear recurrences, the authors of \cite{cd} were able to obtain a density result for all $k$ due to the fact that, appropriately scaling and translating integral linear recurrences, one can again obtain another integral linear recurrence, and thus the problem is reduced to obtaining estimates for $k=1$. Unfortunately, due to the lack of results for dynamical sequences, we can only consider the case of $k=1$. Now, we state our first lemma which would help us in proving Theorem \ref{thm3}.
\begin{lemma}\label{lm7}
Let $G(x)\in\mathbb{Z}[x]$ be a linear polynomial with co-prime coefficients, $z$ be a fixed positive integer and let $$\mathcal{C}_{z}=\{n : p\mid \gcd(a_n,G(n)) \text{ for some prime } p\le z\},$$ then the asymptotic density of $\mathcal{C}_{z}$ exists and
$$\mathbf{d}(\mathcal{C}_{z})\leq  1-\frac{\alpha}{(\log{z})^{1-\gamma}}$$ for some positive constants $\alpha$ and $\gamma<1$.
\end{lemma}
\begin{proof}
Note that $a_{n}$ and $G(n)$ are periodic modulo primes $p$ for which $\ord(p)<\infty.$ Therefore, it is easy to see that $\mathcal{C}_{z}$ is a union of finitely many arithmetic progressions and finite subsets of $\mathbb{N}$; concluding that the density $\mathbf{d}(\mathcal{C}_{z})$ exists.

\noindent Clearly,$$\mathcal{C}_{z}\subseteq\{n : p\mid G(n) \text{ for some prime } p\le z \text{ such that }\ord(p)<\infty\},$$

\vspace{0.1cm}
\noindent so using Eratosthenes' sieve and Lemma \ref{lm1}, we know that
$$\limsup_{x\,\to\,+\infty}\frac{\#\mathcal{C}_{z}(x)}{x}\le 1-\prod_{\substack{p\,\leq\, z\\\ord(p)\,<\,\infty}}\left(1-\frac{1}{p} \right)\le 1-\frac{\alpha}{(\log{z})^{1-\gamma}},$$ for all $z\ge 2$, where $\alpha$ is some positive constant.
\end{proof}

We now prove Theorem \ref{thm3}.
\begin{proof}[Proof of Theorem \ref{thm3}]
The proof follows along the ideas of \cite[Theorem 1.1]{cs} or \cite[Theorem 1.4]{cd}.
For the rest of argument, we assume that all primes $p$ are such that $\ord(p)<\infty.$
Put $\mathcal{C}=\mathbb{N}\setminus\mathscr{A}_{F,\,ax+b,\,1}.$ We need to prove that the asymptotic density of $\mathcal{C}$ exists and is less than $1.$ For each $z>b,$ we split $\mathcal{C}$ into two subsets; $\mathcal{C}_{z}$ and $\mathcal{C}_{z}^{+}=\mathcal{C}\setminus\mathcal{C}_{z}.$ By Lemma~\ref{lm7}, we know that $\mathcal{C}_{z}$ has an asymptotic density. We can see that $\mathbf{d}(\mathcal{C}_{z})$ is a nondecreasing bounded function of $z,$ therefore the limit $$\delta :=\lim_{z\to +\infty}\mathbf{d}(\mathcal{C}_{z})\label{d2}$$ exists and is finite. Thus, we prove that the asymptotic density of $\mathcal{C}$ exists and is equal to $\delta$. If $n\in \mathcal{C}_{z}^+(x),$ then there exists a prime $p>z$ such that $p\mid an+b$ and $p\mid a_n$. Clearly, $p$ is non-anomalous as if $p\mid \gcd(an+b,a_n)$ and $\ord(p)=p$ then $p\mid \gcd(n,an+b)\le b$. Hence, we can write $n=\l(p)\, m$ for some positive integer $m\ll x/\l(p)$ such that $a_{\,\l(p)\, m}\equiv 0\pmod{p}.$ From Lemma~\ref{lm3}-(\ref{lm3.5}), we get that the number of possible values of $m$ is at most $$\OO\left(\frac{x}{p\cdot \ord(p)}+1\right).$$ Therefore, we conclude that $$\#\mathcal{C}_{z}^+(x)\ll \sum_{z\,\le\, p\,\ll\, x}\left(\frac{x}{p\cdot\ord(p)}+1\right)\ll x\cdot\left(\sum_{p>z}\frac{1}{p\cdot \ord(p)}+\frac{1}{\log{x}}\right),$$ where we used Chebyshev's bound for number of primes less than $x.$ Using Lemma \ref{lm5}, we get that $$\frac{\#\mathcal{C}_{z}^+(x)}{x}\ll\frac{1}{\log{z}}+\frac{1}{\log{x}},$$ so that 
\begin{align*}
\limsup_{x \to +\infty} \left|\frac{\#\mathcal{C}(x)}{x} - \mathbf{d}(\mathcal{C}_{z})\right| &=\limsup_{x \to +\infty} \left|\frac{\#\mathcal{C}(x)}{x} - \frac{\#\mathcal{C}_{z}(x)}{x}\right| \\&= \limsup_{x \to +\infty} \frac{\#\mathcal{C}_{z}^+(x)}{x}\ll \frac1{\log{z}} ,\end{align*}

\vspace{0.3cm}
\noindent hence, by letting $z\to\infty,$ we find $\mathbf{d}(\mathcal{C})=\delta$. Now, to compute $\mathbf{d}(\mathcal{C}),$ we have 
\vspace{0.3cm}
\begin{align*}\mathbf{d}(\mathcal{C})=\limsup_{x \to +\infty} \frac{\#\mathcal{C}(x)}{x}&\le\limsup_{x \to +\infty} \frac{\#\mathcal{C}_{z}(x)}{x} + \limsup_{x \to +\infty} \frac{\#\mathcal{C}_{z}^+(x)}{x} \\&\leq 1 - \left(\frac{c_1}{(\log{z})^{1-\gamma}} -\frac{c_2}{\log{z}}\right) ,
\end{align*}

\vspace{0.3cm}
\noindent for all $z \geq 2$, where $c_1$ and $c_2$ are positive constants. Finally, picking a sufficiently large $z,$ depending on $c_1$ and $c_2,$ we get that $\mathbf{d}(\mathcal{C})<1.$ Thus,$$\mathbf{d}(\mathscr{A}_{F,\,ax+b,\,1})=1-\mathbf{d}(\mathcal{C})>0,$$ as desired.\end{proof}

\begin{remark}\label{rm2}
If we restrict $F$ to be one of the following polynomials: 
\begin{enumerate}[(1)]
   \item $F(x)=x^{2}-kx+k$ for some $k\in\mathbb{Z}$
   \item $F(x)=x^2+kx-1$ for some $k\in\mathbb{Z}\setminus\{0,2\}$
   \item $F(x)=x^2+k$ for some $k\in\mathbb{Z}\setminus\{-1\}$
   \item $F(x)=x^2-2xk+k$ for some $k\in\mathbb{Z}\setminus\{\pm 1\},$
\end{enumerate} as  considered in \cite[Theorem 1.2]{rj} and let $G$ be a polynomial with all integral roots and no fixed divisors, then one can prove using arguments in Theorem \ref{thm3} and replacing Lemma \ref{lm14} with results in \cite[Theorem 1.2]{rj} to conclude that $\mathscr{A}_{F,\,G,\,1}$ has a positive asymptotic density.
\end{remark}
\begin{remark}
See that $\mathscr{A}_{F,\,G,\,1}=\mathscr{B}_{F,\,G,\,1}$ and, thus, for $G(x)=ax+b,$ we have a proven version of Theorem \ref{thm1} for all polynomials $F$ that have $0$ as a wandering point and $k=1$. \end{remark}

In light of these results, which regard the sets $\mathscr{A}_{F,\,G,\,1}$ and $\mathscr{B}_{F,\,G,\,1}$, it is natural to ask about the distribution of these sets.

\noindent \textbf{Question.} Can we obtain explicit expressions for asymptotic density for the sets $\mathscr{A}_{F,\,k}$ and $\mathscr{B}_{F,\,k}$ for all polynomials $F$ that have $0$ as a wandering point?

\bibliographystyle{amsplain}

\end{document}